\documentclass[10pt]{article}

\usepackage{amsmath, amsthm, amsfonts, amssymb, mathrsfs}
\usepackage{hyperref}

\newcommand{\N}{\mathbb{N}}
\newcommand{\R}{{\mathbb{R}}}
\newcommand{\C}{{\mathbb{C}}}

\newcommand{\T}{{\mathcal{T}}}
\newcommand{\dd}{{{\rm d}}}
\newcommand{\ii}{{\rm i}}
\renewcommand{\P}{{\mathcal{P}}}
\newcommand{\PT}{{\mathcal{PT}}}

\renewcommand{\H}{{\mathcal{H}}}
\newcommand{\Dom}{{\rm{Dom}\,}}

\newcommand{\mC}{\mathcal{C}}
\newcommand{\mO}{\mathcal{O}}

\renewcommand{\Im}{\text{\rm Im}\,}
\newcommand{\cf}{\emph{cf.}}
\newcommand{\ie}{{\emph{i.e.}}}
\newcommand{\eg}{{\emph{e.g.}}}
\newcommand{\sgn}{\,\mathrm{sgn}}

\begin{document}

%\graphicspath{{Figures/}}
%\graphicspath{{FiguresBW/}}

%
\numberwithin{equation}{section}
\theoremstyle{plain}
\newtheorem{define}{Definition}[section]
\newtheorem{theorem}{Theorem}[section]
\newtheorem{lemma}[theorem]{Lemma}
\newtheorem{proposition}[theorem]{Proposition}
\newtheorem{corollary}[theorem]{Corollary}
\renewcommand{\proofname}{Proof}
\theoremstyle{remark}
\newtheorem{remark}{Remark}[section]

%\newtheorem{example}{Example}
%\theoremstyle{plain}

%-------%
% TITLE %
%-------%

\title{\bf On the similarity of Sturm-Liouville operators
with non-Hermitian boundary conditions to
self-adjoint and normal operators}

\author{David Krej\v ci\v r\'ik$^{a,b}$, \
Petr Siegl,$^{a,c,d}$ \ and \ Jakub \v Zelezn\'y$^{a,c}$}
\date{
% \vspace{4} \\
\small
\emph{
\begin{quote}
\begin{itemize}
\item[$a)$] Department of Theoretical Physics, Nuclear Physics Institute, Czech Academy of Sciences, \v Re\v z, Czech Republic \\
E-mail: david@ujf.cas.cz, siegl@ujf.cas.cz, zelezny@ujf.cas.cz
\item[$b)$] IKERBASQUE, Basque Foundation for Science,
Alameda Urquijo, 36, 5, 48011 Bilbao, Kingdom of Spain
\item[$c)$] Faculty of Nuclear Sciences and Physical Engineering, Czech Technical University in Prague, Prague,  Czech Republic
\item[$d)$] Laboratoire Astroparticule et Cosmologie, Universit\'e Paris 7, Paris, France
\end{itemize}
\end{quote}
}
%\medskip
%\today
24 August 2011
}

\maketitle

\begin{abstract}
\noindent
We consider one-dimensional Schr\"odinger-type operators
in a bounded interval with non-self-adjoint Robin-type
boundary conditions. It is well known that such operators
are generically conjugate to normal operators via
a similarity transformation. Motivated by recent interests
in quasi-Hermitian Hamiltonians in quantum mechanics,
we study properties of the transformations in detail.
We show that they can be expressed as the sum of
the identity and an integral Hilbert-Schmidt operator.
In the case of parity and time reversal boundary conditions,
we establish closed integral-type formulae for
the similarity transformations, derive the similar self-adjoint
operator and also find the associated ``charge conjugation'' operator,
which plays the role of fundamental symmetry
in a Krein-space reformulation of the problem.
\bigskip \\
\noindent
{\bf Mathematics Subject Classification (2010)}: \\
%34B24, 47B40, 34L10, 34L40, 34B09, 81Q12 % Petr
Primary: 34B24, 47B40, 34L10 Secondary: 34L40, 34L05, 81Q12. % David
\end{abstract}

%\newpage
%\tableofcontents

\newpage
\section{Introduction}
Let us consider the m-sectorial realization~$H$
of the second derivative operator
\begin{equation}\label{operator-intro}
  \psi \mapsto -\psi''
\end{equation}
in the Hilbert space $\H := L^2(-a,a)$, with $a>0$,
subjected to separated, Robin-type boundary conditions
\begin{equation}\label{bc-intro}
  \psi'(\pm a) + c_\pm \, \psi(\pm a)=0
\end{equation}
where~$c_{\pm}$ are arbitrary complex numbers.
The operator~$H$ is self-adjoint if, and only if,
the constants~$c_{\pm}$ are real.
The present paper is concerned with the existence and properties
of similarity transformations of~$H$ to a normal
or self-adjoint operator in the non-trivial case
of non-real~$c_\pm$.

The similarity to the normal (respectively, self-adjoint) operator
is understood as the existence of a bounded operator~$\Omega$
with bounded inverse such that
\begin{equation}\label{sim-intro}
  h := \Omega H \Omega^{-1}
\end{equation}
is normal (respectively, self-adjoint).
We remark that this concept is equivalent to the existence of
a topologically equivalent inner product in~$\H$
with respect to which~$H$ is normal (respectively, self-adjoint).
In addition to results on the general structure of
the similarity transformations, modified inner products,
and transformed operators,
we present explicit closed formulae for these objects
in special cases of boundary conditions.

The operators of the type \eqref{operator-intro}--\eqref{bc-intro}
have been studied from many aspects
and there exist a large number of known results;
we particularly mention
the classical monograph of Dunford and Schwartz \cite[Chapter XIX.3]{DS3}.
Recent years brought new motivations and focused attention
to some aspects of the problem which attracted little attention earlier.

As an example, let us mention that one-dimensional Schr\"odinger operators
with non-Hermitian boundary conditions of the type~\eqref{bc-intro}
were used as a model in semiconductor physics
by Kaiser, Neidhardt and Rehberg \cite{Kaiser-2003-252}.
In their paper the imaginary parts of the constants~$c_\pm$ are required to
have opposite signs such that the system is dissipative.
The authors find the characteristic function of the operators,
construct its minimal self-adjoint dilation
and develop the generalized eigenfunction expansion for the dilation.
See also \cite{Kaiser-2002-43,Kaiser-2003-45}
for further generalizations.
Here the main idea of using non-self-adjointness
comes from embedding a quantum-mechanically
described structure into a macroscopic flow
and regarding the system as an open one.

However, the principal motivation of the present work
is the possibility of giving
a direct quantum-mechanical interpretation of non-Hermitian operators
which are similar to self-adjoint ones \cite{Scholtz-1992-213}.
The most recent strong impetus to this point of view
comes from the so-called $\PT$-symmetric quantum mechanics.
Here the reality of the spectrum of a class of non-Hermitian operators
-- caused by certain symmetries rather than self-adjointness --
suggests their potential relevance as quantum-mechanical Hamiltonians;
see the review articles \cite{Bender-2007-70,Mostafazadeh-2010-7}.
It has been confirmed during the last years that it is indeed the case
provided that the similarity transformation to a self-adjoint
operator can be ensured. However, it is a difficult task.

Motivated by the lack of rigorous results,
the authors of \cite{Krejcirik-2006-39} introduced
a simple non-Hermitian $\PT$-symmetric operator
of the type \eqref{operator-intro}--\eqref{bc-intro}
and wrote down a closed formula for the (square of the)
similarity transformation
(see also \cite{Krejcirik-2008-41a,Krejcirik-2010-43}).
Let us also mention that the importance of (not only) $\PT$-symmetric
version of \eqref{operator-intro}--\eqref{bc-intro}
in quantum mechanical scattering has been recently established
in \cite{Hernandez-Coronado-2011-375}.

The present paper can be regarded as a step further.
In addition to considering more general situations
of larger classes of boundary conditions
and similarity to normal operators, we provide an alternative
and more elegant (integral-kernel) formulae for
the similarity transformations in the $\PT$-symmetric situation.
Moreover, we also give a remarkably simple formula
for the similar self-adjoint operator in this case.
Finally, we succeed in finding
the so-called $\mathcal{C}$-operator in a closed form,
which plays the role of fundamental symmetry
in a Krein-space reformulation of the problem.

The paper is organised as follows.
In Section~\ref{sec:prel}
we give a precise definition of the operator~$H$,
summarize its known properties
and recall the general concepts of quasi-Hermitian,
$\PT$-symmetric, and $\mathcal{C}$-symmetric operators.
Our main results about the universal structure
of the similarity transformations
can be found in Section~\ref{sec:gen.res}.
In Section~\ref{sec:examples} we show how these
can be applied to particular ($\PT$-symmetric)
classes of boundary conditions
and we present some explicit constructions of the studied objects.
In Section~\ref{sec:bound.pert} we discuss how the results
can be extended to bounded
and even second-order perturbations of~$H$.
Our final Section~\ref{Sec.end} presents a series
of concluding remarks.
%

%----------------------%
\section{Preliminaries}
\label{sec:prel}
%----------------------%
%
We start with recalling general properties of~$H$
and concepts of similarity transformations in Hilbert spaces.

%--------------------------------------------------------------%
\subsection{Definition of the operator \texorpdfstring{$H$}{H}}
%--------------------------------------------------------------%
%
The standard norm in our Hilbert space
$\H \equiv L^2(-a,a)$ is denoted by $\|\cdot\|$.
The corresponding inner product is denoted by $\langle\cdot,\cdot\rangle$
and it is assumed to be antilinear in the first component.

We consider the m-sectorial realization~$H$
of the operator~\eqref{operator-intro}
subjected to the boundary conditions~\eqref{bc-intro}
as the operator associated on~$\H$
with the quadratic form
\begin{equation}\label{tH.def}
\begin{aligned}
t_H[\psi]&:= \|\psi'\|^2 + c_+ |\psi(a)|^2 - c_- |\psi(-a)|^2,  \\
\Dom(t_H)&:=W^{1,2}(-a,a).
\end{aligned}
\end{equation}
Note that the boundary terms are well defined
because of the embedding of the Sobolev space $W^{1,2}(-a,a)$
in the space of uniformly continuous functions $C^0[-a,a]$.
An elementary idea of the proof of the embedding can be
also used to show that the boundary terms represent
a relatively bounded perturbation of the form
associated with the Neumann Laplacian (\ie, $c_\pm=0$).
Since the Neumann form is clearly non-negative
and closed by definition of the Sobolev space,
we know that~$t_H$ is a closed sectorial form
by a standard perturbative argument \cite[Sec.~VI.1.6]{Kato-1966}.

By the representation theorem \cite[Thm.~VI.2.1]{Kato-1966}
and an elementary version of standard elliptic regularity theory,
it is easy to see that
\begin{equation}\label{H.def}
\begin{aligned}
H \psi & = -\psi'', \\
\Dom(H) & = \big\{ \psi \in W^{2,2}(-a,a): \
\psi'(\pm a) + c_\pm \psi(\pm a)=0 \big\}.
\end{aligned}
\end{equation}
We refer to \cite[Ex. VI.2.16]{Kato-1966} for more details.
The operator definition~\eqref{H.def} gives a precise meaning
to \eqref{operator-intro}--\eqref{bc-intro}.

%-----------------------------------------------------%
\subsection{Dirichlet and Neumann boundary conditions}\label{subsec.not}
%-----------------------------------------------------%
%
This subsection is mainly intended to collect some
notation we shall use later.

We have already mentioned that the special choice $c_\pm=0$
gives rise to the Neumann Laplacian $-\Delta_N$ on~$\H$.
The Dirichlet Laplacian $-\Delta_D$ on~$\H$ can be considered
as the other extreme case by formally putting $c_\pm=+\infty$.
It is properly defined as the second derivative operator~\eqref{operator-intro}
with the operator domain
$
  \Dom(-\Delta_D) :=
  W^{2,2}(-a,a) \cap W_0^{1,2}(-a,a)
$.

The spectrum of the Dirichlet and Neumann Laplacians
in our one-dimen\-sion\-al situation is well known:
$$
\begin{aligned}
  \sigma(-\Delta_D) &= \{k_n^2\}_{n=1}^{\infty}
  \,, \\
  \sigma(-\Delta_N) &= \{k_n^2\}_{n=0}^{\infty}
  \,,
\end{aligned}
  \qquad\mbox{with}\qquad
  k_n :=\frac{n\pi}{2a}
  \,.
$$
The corresponding eigenfunctions are respectively given by
\begin{equation}\label{not.1}
\begin{aligned}
  \chi_n^D(x) :=\frac{1}{\sqrt{a}} \sin k_n(x+a),
  \qquad
  \chi_n^N(x) :=
  \begin{cases}
    \frac{1}{\sqrt{2a}} & \mbox{if} \ n=0 \,,
    \\
    \frac{1}{\sqrt{a}} \cos k_n(x+a) & \mbox{if} \ n \geq 1 \,.
  \end{cases}
\end{aligned}
\end{equation}
To simplify some expressions in the sequel,
we extend the notation by $\chi_0^D:=0$.

Next we introduce a ``momentum" operator $p$ and its adjoint $p^*$:
\begin{equation}\label{pp*.def}
\begin{aligned}
  p\psi   &:=-\ii \psi', & \quad
  p^*\psi &=-\ii \psi',
  \\
  \Dom(p)   &:= W_0^{1,2}(-a,a), & \quad
  \Dom(p^*) &= W^{1,2}(-a,a).
\end{aligned}
\end{equation}
The following identities hold:
\begin{equation}\label{p.ND.id}
\begin{aligned}
\ii p \chi_n^D & = k_n \chi_n^N, \quad & \ii p^* \chi_n^N & = - k_n \chi_n^D, \\
-\Delta_D  & =  p^*p, \quad &  -\Delta_N  & =  pp^*.
\end{aligned}
\end{equation}

The resolvents $(-\Delta_D-k^2)^{-1}$, $(-\Delta_N-k^2)^{-1}$
act as integral operators with simple kernels (Green's functions)
$\mathcal{G}_D^k$ and $\mathcal{G}_N^k$,
respectively:
\begin{equation}\label{res.DN.ker}
\begin{aligned}
%\mathcal{G}_D(x,y) &= \frac{-1}{{k \sin(2ka)}}
%\begin{cases}
%\sin(k(x+a)) \sin(k(y-a)), & x<y, \\
%\sin(k(y+a)) \sin(k(x-a)), & x>y,
%\end{cases}
\mathcal{G}_D^k(x,y) &= \frac{- \sin(k(x+a)) \, \sin(k(y-a))}{k \;\! \sin(2ka)} \,,
& x<y \,, &
\\
%\mathcal{G}_N(x,y) &= \frac{-1}{{k \sin(2ka)}}
%\begin{cases}
%\cos(k(x+a)) \cos(k(y-a)), & x<y, \\
%\cos(k(y+a)) \cos(k(x-a)), & x>y,
%\end{cases}
\mathcal{G}_N^k(x,y) &= \frac{- \cos(k(x+a)) \, \cos(k(y-a))}{k \;\! \sin(2ka)} \,,
& x<y \,, &
\end{aligned}
\end{equation}
with $x,y$ exchanged for $x>y$.
Here the spectral parameter~$k^2$ is supposed to belong to
the resolvent set of the respective operator.

For $k=0$, the kernel of $(-\Delta_D)^{-1}$ simplifies to
\begin{equation}\label{res.D0}
\begin{aligned}
\mathcal{G}_D^0(x,y) &= \frac{(x+a)(a-y)}{2a}\,, & x<y& \,,
\end{aligned}
%\begin{cases}
%(x+a)(a-y), & x<y, \\
%(y+a)(a-x), & x>y.
%\end{cases}
\\
\end{equation}
with $x,y$ exchanged for $x>y$.
The resolvent of $-\Delta_N$ does not exist for $k=0$, of course,
but one can still introduce the reduced resolvent
$\left(-\Delta_N^{\perp}\right)^{-1}$
of the Neumann Laplacian with respect to the eigenvalue~$0$
(see~\cite[Sec.~III.6.5]{Kato-1966} for the concept
of reduced resolvent).
From the point of view of the spectral theorem:
\begin{equation}
\left(-\Delta_N^{\perp}\right)^{-1}
= \sum_{n=1}^{\infty} \frac{1}{k_n^2} \, \chi_n^N \langle \chi_n^N, \cdot \rangle.
\end{equation}
The corresponding integral kernel $\mathcal{G}_N^{\perp}(x,y)$
can be obtained by taking the limit
$k \to 0$ of the regularized expression
$
  G_N^k(x,y)+k^{-2}\;\!\chi_0^N(x)\chi_0^N(y)
$.
We find
\begin{equation}\label{res.N.r}
\begin{aligned}
\mathcal{G}_N^{\perp}(x,y)
&= \frac{(x+a)^2}{4a} + \frac{(y-a)^2}{4a} - \frac{a}{3}\,,
&  x<y \,, &
\end{aligned}
%\mathcal{G}_N^{\perp}(x,y) = \frac{x^2+y^2}{4a}+\frac{x+y}{2}+\frac{a}{6} -
%\begin{cases}
%y, & x<y, \\
%x, & x>y.
%\end{cases}
\end{equation}
with $x,y$ exchanged for $x>y$.

Finally, we introduce operators
\begin{equation}\label{JND.def}
  J^{\iota}:=
  \sum_{n=0}^{\infty} C_n^2 \,
  \chi_n^{\iota} \langle \chi_n^{\iota}, \cdot \rangle\,,
  \qquad
  \iota \in \{D,N\}
  \,,
\end{equation}
where $C_n$ are positive numbers satisfying
\begin{equation}\label{Cn.def}
  0 < m_1 < C_n < m_2 < \infty
\end{equation}
for all $n \geq 0$, with given positive $m_1, m_2$.
The sum in the definition \eqref{JND.def}, as well as
all other analogous expressions in the following,
are understood as limits in the strong sense.

%---------------------------------------------------------%
\subsection{General properties of \texorpdfstring{$H$}{H}}
%---------------------------------------------------------%
%
Now we are in a position to recall some general properties
of the operator~$H$.
\pagebreak[2]
\begin{proposition}[General known facts]\label{H.spectral}%
\noindent%
\begin{enumerate}%
\item[\emph{(i)}]
$H$ is m-sectorial.
The adjoint operator $H^*$ is obtained by taking
the complex conjugation of $c_{\pm}$
in the boundary conditions~\eqref{bc-intro}.
\item[\emph{(ii)}]
$H$ forms a holomorphic family of operators of type~$(B)$
with respect to the boundary parameters~$c_{\pm}$.
\item[\emph{(iii)}]
The resolvent of~$H$ is a compact operator.
\item[\emph{(iv)}]
$H$ is a discrete spectral operator.
\item[\emph{(v)}]
If all eigenvalues are simple,
then $H$ is similar to a normal operator.
If the spectrum of $H$ is in addition real,
then $H$ is similar to a self-adjoint operator.
\end{enumerate}
\end{proposition}

We have already shown that~$H$ is m-sectorial
as the operator associated with
the closed sectorial form~\eqref{H.def}.
The rest of the claim~(i) follows by the fact that
the adjoint operator is associated
with the adjoint form $t_H^*(\phi,\psi) := \overline{t_H(\psi,\phi)} $
(\cf~\cite[Thm.~VI.2.5]{Kato-1966}).
Property~(ii) follows from~\eqref{H.def} as well
if we recall that the boundary terms represent
a relatively bounded perturbation of the form
associated with the Neumann Laplacian
and the relative bound can be made arbitrarily small
(\cf~\cite[Sec.~VII.4.3]{Kato-1966}).
This also proves~(iii) as a consequence of
the perturbation result \cite[Sec.~VI.3.4]{Kato-1966}.
The proof of~(iv) is contained in \cite[Chapter XIX.3]{DS3}.
Property~(v) is a consequence of~(iv).

The similarity to a normal operator can be equivalently stated
as the Riesz basicity of the eigenvectors of~$H$.
This property is shared by all second derivative operators
with strongly regular boundary conditions, see \cite{Mikhajlov-1962-3}.
Using the notion of spectral operator, this has been investigated
in~\cite{DS3} as well.

Although the eigenvalues of~$H$ are generically simple,
degeneracies may appear. However, the only possibility are the eigenvalues
of algebraic multiplicity two and geometric multiplicity one.
In this case, operator $H$ cannot be similar to a normal one,
nevertheless, the eigenvectors together with generalized eigenvectors
still form a Riesz basis.

Now we turn to symmetry properties of~$H$.
\begin{define}[$\PT$-symmetry]\label{Def.PT}
We say that~$H$ is $\PT$-symmetric if
\begin{equation}\label{def.PT}
[\PT, H]=0,
\end{equation}
where
\begin{equation}
(\P\psi)(x):=\psi(-x), \ \ \ (\T\psi)(x):=\overline{\psi(x)}.
\end{equation}
\end{define}

It should be stressed that $\PT$ is an antilinear operator.
The commutator relation~\eqref{def.PT} means precisely
that $(\PT)H \subset H (\PT)$,
as usual for the commutativity of an unbounded operator
with a bounded one (\cf~\cite[Sec.~III.5.6]{Kato-1966}).
In the quantum-mechanical context,
$\P$~corresponds to the parity inversion (space reflection),
while~$\T$ is the time reversal operator.

\begin{define}[$S$-self-adjointness]\label{Def.Jsa}
We say that~$H$ is $S$-self-adjoint
if the relation $H=S^{-1} H^* S$ holds
with a boundedly invertible operator~$S$.
\end{define}
We will use this concept in a wide sense,
with~$S$ being either linear or antilinear operator.
If~$S$ is a conjugation operator (\ie~antilinear involution),
then our definition coincides with the concept of $J$-self-adjointness
\cite[Sec.~III.5]{EE}.

While Definition~\ref{Def.Jsa} is quite general,
Definition~\ref{Def.PT} makes sense for operators
in a complex functional Hilbert space only.
In our case, we have:
\begin{proposition}[Symmetry properties]\label{Prop.symmetry}
\noindent
\begin{enumerate}
\item[\emph{(i)}]
$H$ is $\mathcal{T}$-self-adjoint.
\item[\emph{(ii)}]
$H$ is $\mathcal{P}$-self-adjoint if, and only if,
$c_{-}=-\overline{c_+}$.
\item[\emph{(iii)}]
$H$ is $\PT$-symmetric if, and only if,
$c_{-}=-\overline{c_+}$.
\end{enumerate}
\end{proposition}

Property (ii) coincides with the notion of self-adjointness
in the Krein space equipped with
the indefinite inner product $\langle \cdot, \P \cdot \rangle$.
It is also referred to as $\P$-pseudo-Hermiticity in physical literature
(see, \eg, \cite{Mostafazadeh-2010-7}).

It follows from Proposition~\ref{Prop.symmetry}.(i) that the residual
spectrum of~$H$ is empty (\cf~\cite[Corol.~2.1]{Borisov-2008-62}).
Alternatively, it is a consequence of Proposition~\ref{H.spectral}.(iii),
which in addition implies that the spectrum of~$H$ is purely discrete.

We denote the (countable) set of eigenvalues of~$H$
by $\{\lambda_n\}_{n=0}^\infty$ and the corresponding
set of eigenfunctions by $\{\psi_n\}_{n=0}^\infty$.
Similarly, let $\{\overline{\lambda_n}\}_{n=0}^\infty$
and $\{\phi_n\}_{n=0}^\infty$ be the set of eigenvalues
and eigenfunctions of the adjoint operator~$H^*$.
That is
\begin{equation}\label{H.psi.phi}
H \psi_n = \lambda_n \psi_n, \qquad
H^* \phi_n = \overline{\lambda_n} \phi_n.
\end{equation}
Eigenfunctions $\psi_n$ and $\phi_m$
corresponding to different eigenvalues,
\ie\ $\lambda_n\not=\lambda_m$,
are clearly orthogonal.
Solving the eigenvalue equation for~$H$
in terms of sine and cosine functions,
it is straightforward to reduce the boundary value problem
to an algebraic one.
\begin{proposition}[Spectrum]\label{EV.EF}
\noindent
The eigenvalues $\lambda_n=l_n^2$ of $H$ are solutions
of the implicit equation
\begin{equation}\label{EV.eq}
\sin(2 a l) (c_- c_+ + l^2) + (c_- - c_+)l \cos(2 a l)=0
\,.
\end{equation}
The corresponding eigenfunctions of $H$ and $H^*$ respectively read
\begin{equation}\label{H.EF}
\begin{aligned}
\psi_n(x) &
= A_n \frac{1}{\sqrt{a}} \left( \cos (l_n(x+a)) - \frac{c_-} {l_n}
\sin(l_n(x+a)) \right), & \\
\phi_n(x) & = \frac{1}{\sqrt{a}}
\left( \cos (\overline{l_n}(x+a)) - \frac{\overline{c_-}}
{\overline{l_n}} \sin(\overline{l_n}(x+a)) \right).
\end{aligned}
\end{equation}
If all eigenvalues are simple, $\psi_n $ can be normalized
through the coefficients $A_n$ in such a way that
$\langle \psi_n, \phi_m \rangle = \delta_{nm}$.
\end{proposition}

The spectrum of $H$ has been described more explicitly
for the $\PT$-symmetric case.
First of all, as a consequence of the symmetry,
we know that the spectrum is symmetric with respect to
the real axis.
In the following proposition we summarize more precise results
obtained in \cite{Krejcirik-2006-39, Krejcirik-2010-43}.
\begin{proposition}[$\PT$-symmetric spectrum]\label{H.PT.EV}
Let $c_{\pm}=\ii \alpha \pm \beta$, with $\alpha,\beta \in \R$.
\begin{enumerate}
\item
If $\beta=0$ then all eigenvalues of $H$ are real,
\begin{equation}\label{H.beta.0.EV}
\lambda_0={\alpha^2}, \ \ \lambda_n=k_n^2, \ \ n \in \N.
\end{equation}
The corresponding eigenfunctions of $H$ and $H^*$
respectively read
\begin{equation}\label{H.beta.0.EF}
\begin{aligned}
\psi_0(x) &= A_0 e^{-\ii \alpha (x+a)}, &
\psi_n(x)& =A_n \left (\chi_n^N(x) -\ii \frac{\alpha} {k_n} \chi_n^D(x) \right), & \\
\phi_0(x) &= \frac{1}{\sqrt{2a}}\, e^{\ii \alpha (x+a)}, &
\phi_n(x)& = \chi_n^N(x) + \ii \frac{\alpha} {k_n} \chi_n^D(x). & \\
\end{aligned}
\end{equation}
If $\alpha\neq k_n$ for every $n \in \mathbb{N}$,
then all the eigenvalues are simple
and choosing
\begin{equation}\label{H.beta.0.EF2}
\begin{aligned}
A_0&:=\frac{\alpha e^{2\ii \alpha a} \sqrt{2a}}{ \sin (2 \alpha a)},
& A_n:=\frac{k_n^2}{k_n^2-\alpha^2},
\end{aligned}
\end{equation}
we have the biorthonormal relations
$\langle \psi_n, \phi_m \rangle = \delta_{nm}$.

\item
If $\beta>0$, then all the eigenvalues of $H$ are real and simple.

\item
If  $\beta<0$, then all the eigenvalues are either real
or there is one pair of complex conjugated eigenvalues
with real part located in the neighborhood of $\alpha^2+\beta^2$.
\end{enumerate}
In any case,
the eigenvalue equation \eqref{EV.eq} can be rewritten as
\begin{equation}\label{PT.EV.eq}
(l^2-\alpha^2-\beta^2) \sin (2a l) - 2 \beta l \cos(2a l )=0.
\end{equation}
\end{proposition}
%

%------------------------------------------%
\subsection{Concept of the metric operator}\label{metric.c}
%------------------------------------------%
%
We recall the concept of metric operator
(or quasi-Hermitian operators introduced in \cite{Dieudonne-1961}),
widely used in $\PT$-symmetric literature.
\begin{define}[Metric operator and quasi-Hermiticity]\label{Theta.definition}
Bounded positive
\footnote{$A$ is positive if $\langle f, A f \rangle > 0$
for all $f \in \H,$ $f\neq 0$.}
operator $\Theta$ with bounded inverse
is called a metric operator for $H$,
if $H$ is $\Theta$-self-adjoint.
$H$ is then called quasi-Hermitian.
\end{define}

It is obvious that the quasi-Hermitian operator~$H$
is self-adjoint with respect to the modified inner product
$\langle \cdot , \cdot \rangle_{\Theta}:=\langle \cdot, \Theta \cdot \rangle$.
It is also not difficult to show that the metric operator
exists if, and only if, $H$~is similar to a self-adjoint operator.
Moreover, since~$H$ has purely discrete spectrum,
the metric operator can be obtained as
\begin{equation}\label{Theta.def}
\Theta=\sum_{n=0}^{\infty} C_n^2 \, \phi_n \langle \phi_n, \cdot \rangle,
\end{equation}
where $\phi_n$ are eigenfunctions of $H^*$
and $C_n$ are real constants satisfying \eqref{Cn.def}.

The expression \eqref{Theta.def} illustrates a non-uniqueness
of the metric operator caused by the arbitrariness of~$C_n$.
The latter can be actually viewed as a modification of
the normalization of functions $\phi_n$.
Choosing different sequences $\{C_n\}_{n=0}^{\infty}$,
we obtain all metric operators for $H$, \cf~\cite{Siegl-MT, Siegl-2011-50}.

It is important to stress that if we define an operator $\Theta$
by \eqref{Theta.def}, we find that such $\Theta$ is bounded, positive,
and with bounded inverse whenever $\{\phi_n\}_{n=0}^{\infty}$
is a Riesz basis. Thus, by virtue of Proposition~\ref{H.spectral}.(v),
such a $\Theta$ exists if, and only if, all eigenvalues of $H$ are simple.
However, the $\Theta$-self-adjointness of $H$ is satisfied
if, and only if, the spectrum of $H$ is real.
Otherwise, only $\Theta H \Theta^{-1}H^* = H^* \Theta H \Theta^{-1}$ holds,
\cf~\cite{Siegl-2011-50}, which is equivalent to the fact
that $H$ is similar to a normal operator.

In the following, the operator $\Theta$ is always defined by~\eqref{Theta.def}
regardless if it is a metric operator for $H$
in view of Definition \ref{Theta.definition}.

It should be also noted that $\Theta$, as a positive operator,
can be always decomposed to
\begin{equation}\label{Theta.dec}
\Theta = \Omega^* \Omega.
\end{equation}
One example of such $\Omega$ is obviously $\sqrt{\Theta}$.
We shall take the advantage of some different decompositions
of the type~\eqref{Theta.dec} later.
It follows easily from Definition~\ref{Theta.definition}
that the similar operator~$h$ defined by~\eqref{sim-intro}
with~$\Omega$ given by~\eqref{Theta.dec}
is self-adjoint if~$\Theta$ is a metric operator for~$H$.
If all eigenvalues of $H$ are simple but no longer entirely real,
$h$~is (only) a normal operator.

%---------------------------------------------------------------------%
\subsection{Concept of the \texorpdfstring{$\mathcal{C}$}{C} operator}
%---------------------------------------------------------------------%
%
For $\PT$-symmetric operators, the notion of $\mC$ operator
was introduced in \cite{Bender-2002-89}
and formalized in \cite{Albeverio-2005-38}.
It was observed in \cite{Langer-2004-54}
and in many works after that paper
that Krein spaces provide suitable
framework for studying $\PT$-symmetric operators.
Indeed, $\PT$-symmetric operators
which are at the same time $\mathcal{P}$-self-adjoint
are in fact self-adjoint in the Krein space equipped
with the indefinite inner product
$\langle \cdot,\P\cdot \rangle$.
Recall that our operator $H$ is $\P$-self-adjoint
if, and only if, it is $\PT$-symmetric
(\cf~Proposition~\ref{Prop.symmetry}).
\begin{define}[$\mathcal{C}$ operator]\label{c.def}
Assume that $H$ is $\P$-self-adjoint (\cf~Proposition~\ref{Prop.symmetry}).
We say that $H$
possesses the property of $\mC$-symmetry,
if there exists a bounded linear operator $\mC$
such that $[H,\mC] = 0,$ $\mC^2=I,$ and $\P \mC $ is a metric operator for~$H$.
\end{define}
Thus, from the point of view of metric operators,
we can find the $\mC$~operator
as $\mC:=\P \Theta$ for $\Theta$ satisfying $(\P \Theta)^2 = I$.
Hence $\mC$-symmetry allows us to naturally choose a metric operator.
Besides a possible physical interpretation of $\mC$ discussed
in \cite{Bender-1999-40, Bender-2007-70},
it appears naturally in the Krein spaces framework as pointed out
in \cite{Kuzhel-2010-16,Kuzhel-2011-379} as a fundamental symmetry
of the Krein space $(\H, \langle \cdot, \P\cdot \rangle)$
with an underlying Hilbert space $(\H, \langle \cdot,\P \mC \cdot \rangle )$.

%---------------------------------------------%
\section{General results}\label{sec.metric}
\label{sec:gen.res}
%---------------------------------------------%
%
In this section we provide general properties of
the metric operator~$\Theta$ defined in \eqref{Theta.def}
and its decompositions~$\Omega$ from~\eqref{Theta.dec}.

Let $\{\psi_n\}_{n=0}^{\infty}$ and $\{\phi_n\}_{n=0}^{\infty}$
denote the set of eigenvectors of~$H$ and~$H^*$, respectively.
We assume that~$\psi_n$ and~$\phi_n$ form Riesz bases
and that they are normalized in such a way
that $\langle \psi_n, \phi_m \rangle =\delta_{mn}$.
In view of Propositions~\ref{H.spectral}, \ref{EV.EF},
we know that this is satisfied if all the eigenvalues of~$H$ are simple,
which is a generic situation.

Let $\{e_n\}_{n=0}^{\infty}$ be any orthonormal basis of~$\H$.
If all eigenvalues of $H$ are simple, we introduce an operator $\Omega$ by
\begin{equation}\label{Omega.def}
\Omega := \sum_{n=0}^{\infty} e_n \langle \phi_n, \cdot \rangle.
\end{equation}
Clearly, $\Omega:\psi_n \mapsto e_n$.

$\Omega$ is defined only if all eigenvalues are simple, however, sometimes it is possible to
extend it by continuity, see examples in Section \ref{sec:examples}. Nonetheless, such $\Omega$ is typically not invertible and the dimension of the kernel corresponds to the size of Jordan blocks appearing in the spectrum of $H$.

Basic properties of~$\Omega$ are summarized in the following.
\begin{lemma}
Let all eigenvalues of $H$ be simple.
Then $\Omega$ is a bounded operator with bounded inverse
given by
\begin{equation}\label{Omega.inv.def}
\Omega^{-1}= \sum_{n=0}^{\infty} \psi_n \langle e_n, \cdot \rangle,
\end{equation}
\ie~$\Omega^{-1}: e_n \mapsto \psi_n$.
The adjoint of $\Omega$ reads
\begin{equation}\label{Omega.adj.def}
\Omega^*=\sum_{n=0}^{\infty} \phi_n \langle e_n, \cdot \rangle.
\end{equation}
\ie~$\Omega^*: e_n \mapsto \phi_n$
and $\Omega^*\Omega=\Theta$, where $\Theta$ is defined in \eqref{Theta.def} with $C_n=1$.
\end{lemma}
Furthermore, we show how the operator $\Omega$ can be realized.
\begin{theorem}\label{Omega.real}
Let all eigenvalues of $H$ be simple. $\Omega$ can be expressed as
\begin{equation}\label{Omega.real.1}
\Omega = U + L,
\end{equation}
where $U := \sum_{n=0}^{\infty} e_n \langle \chi_n^N, \cdot \rangle$, \ie~$U:\chi_n^N \mapsto e_n$, is a unitary operator, and $L$ is a Hilbert-Schmidt operator.
\end{theorem}
\begin{proof}
At first we remark that it suffices to prove that $\Omega=I+\tilde{L}$
for $e_n:=\chi_n^N$, where $\tilde{L}$ is Hilbert-Schmidt.
More precisely, if we compose $U$ from the claim and $I+\tilde{L}$,
we obtain $\Omega$ in \eqref{Omega.real.1}
since $L=U\tilde{L}$ is Hilbert-Schmidt too.
Thus, we consider this choice of~$e_n$ in the following.
Furthermore, we put $a:=\pi/2$ for simplification of the formulae.
This specific choice is in fact harmless, since we
can easily transfer the results for different $a$ using the isometry
$V:L^2(-\pi/2,\pi/2) \rightarrow L^2(-a,a)$
defined by
$\psi(x) \mapsto \sqrt{\frac{\pi}{2a}} \psi(\frac{\pi x}{2a})$.

The asymptotic analysis of eigenvalues of $H$ in
\cite[proof of Lem.~XIX.3.10]{DS3} shows that
\begin{equation}\label{ln.as}
\begin{aligned}
l_n &= n+\frac{c_+ - c_-}{\pi n} + \mO(n^{-2}),\\
\lambda_n  \equiv l_n^2 & = k_n^2 + \frac{2 (c_+ - c_-)}{\pi} + \mO(n^{-1}),
\end{aligned}
\end{equation}
and $|\Im(l_n)|$ is uniformly bounded in $n$.
These formulae are valid except for a finite number~$N_0$
of eigenvalues.

We set $\varepsilon_n:= l_n-k_n=l_n-n$.
Using elementary trigonometric identities,
we rewrite the eigenfunctions~$\phi_n$ as follows
\begin{equation}\label{phi.rewritten}
\begin{aligned}
\phi_n(x) &= \chi_n^N(x)
\cos \left( \overline{\varepsilon_n} (x+a) \right) -
\chi_n^D(x) \sin \left(\overline{\varepsilon_n} (x+a)\right)  \\
& \quad - \frac{\overline{c_-}}{\overline{l_n}}
\left[
\chi_n^D(x) \cos (\overline{\varepsilon_n} (x+a) ) +
\chi_n^N(x) \sin \left( \overline{\varepsilon_n} (x+a) \right)
\right].
\end{aligned}
\end{equation}
We further rewrite the cosine and sine functions in this expression as
\begin{equation}\label{cn.sn.def}
\begin{aligned}
\cos \left(\overline{\varepsilon_n} (x+a)\right)
&= 1 + \overline{\varepsilon_n}^2 \,
\frac{\cos \left(\overline{\varepsilon_n} (x+a)\right) -1}
{\overline{\varepsilon_n}^2}
=:1 + \overline{\varepsilon_n}^2 \, c_n(x), \\
\sin \left(\overline{\varepsilon_n} (x+a)\right)
&= \overline{\varepsilon_n} \,
\frac{\sin \left(\overline{\varepsilon_n} (x+a)\right)}{\overline{\varepsilon_n}}
=:\overline{\varepsilon_n} \, s_n(x). \\
\end{aligned}
\end{equation}
Note that $\|c_n\|$ and $\|s_n\|$
are uniformly bounded in $n$
because of the properties of $\varepsilon_n$.
The building block $\chi_n^N \langle \phi_n, \cdot \rangle $ of $\Omega$
then becomes
\begin{equation}
\begin{aligned}
\chi_n^N \langle \phi_n, \cdot \rangle & =  \chi_n^N \langle \chi_n^N, \cdot \rangle +
\varepsilon_n^2 \chi_n^N \langle \chi_n^N c_n, \cdot \rangle -
\varepsilon_n \chi_n^N \langle \chi_n^D s_n, \cdot \rangle \\
& \quad - \frac{c_-}{l_n}
\left( \chi_n^N \langle \chi_n^D, \cdot \rangle + \varepsilon_n^2 \chi_n^N \langle \chi_n^D c_n, \cdot \rangle
+ \varepsilon_n \chi_n^N \langle \chi_n^N s_n, \cdot \rangle
\right).
\end{aligned}
\end{equation}
Taking the sum of $\chi_n^N \langle \phi_n, \cdot \rangle $
as in \eqref{Omega.def},
we obviously get $\Omega = I+\tilde{L}$.

It remains to show that the Hilbert-Schmidt norm $\|\tilde{L}\|_{\rm HS}$
of $\tilde{L}$ is finite. We will understand $\tilde{L}$
as a sum $\tilde{L}=\tilde{L}_{N_0} +  \tilde{L}_{\infty}$, where
\begin{equation}
\tilde{L}_{N_0}:= \sum_{n=0}^{N_0-1} \chi_n^N \langle \tilde{\phi}_n, \cdot \rangle, \ \
\tilde{L}_{\infty}:= \sum_{n=N_0}^{\infty} \chi_n^N \langle \tilde{\phi}_n, \cdot \rangle,
\end{equation}
and $\tilde{\phi}_n:=\phi_n-\chi_n^N$.
$\tilde{L}_{N_0}$ is a finite rank operator,
hence it is automatically Hilbert-Schmidt
and it suffices to consider $\tilde{L}_{\infty}$
in the rest of the proof.
We estimate explicitly
only one term in the expression for $\|\tilde{L}_{\infty}\|^2_{\rm HS}$, the rest follows in a similar way:
\begin{equation}
\begin{split}
&\sum_{p=0}^{\infty}
 \left\langle
\sum_{n=N_0}^{\infty} \varepsilon_n \chi_n^N \left\langle \chi_n^D s_n, \chi_p^N   \right \rangle,
\sum_{m=N_0}^{\infty} \varepsilon_m \chi_m^N \left\langle \chi_m^D s_m, \chi_p^N   \right \rangle
\right  \rangle    \\
&= \sum_{p=0}^{\infty} \sum_{n=N_0}^{\infty} |\varepsilon_n|^2 \left| \left\langle \chi_n^D s_n, \chi_p^N \right  \rangle \right|^2 \leq
\frac{1}{a} \sum_{n=N_0}^{\infty} |\varepsilon_n|^2  \| s_n\|^2 < \infty.
\end{split}
\end{equation}
Here the first inequality follows by the Bessel inequality
(after interchanging the order of summation, which is justified)
and by estimating~$\chi_n^D$ by its supremum norm.
The asymptotic behavior of $\varepsilon_n$
and the uniform boundedness of $\|s_n\|$ are used in the last step.
\end{proof}
\begin{corollary}
Let all eigenvalues of $H$ be simple. Then
\begin{equation}\label{Theta.dec.real}
\Theta:= \Omega^* \Omega = I + K
\end{equation}
coincides with $\Theta$ defined in \eqref{Theta.def} with $C_n=1$.
Here~$K$ is a Hilbert-Schmidt operator
that can be realized as an integral operator
with a kernel belonging to $L^2((-a,a)\times(-a,a))$.
\end{corollary}
\begin{proof}
The claim follows from Theorem \ref{Omega.real} and the well-known
facts that Hilbert-Schmidt operators are *-both-sided ideal
in the space of bounded operators and can be realized as integral ones,
see \cite[Thm.VI.23]{Reed1}.
\end{proof}
\begin{remark}
Slight modification of the definition of $\Omega$
and the proof of Theorem~\ref{Omega.real} yields the analogous result
for operators $\Theta$ defined in \eqref{Theta.def} with arbitrary $C_n$.
It suffices to consider $f_n:=C_n e_n$ instead of $e_n$.
The resulting form is
\begin{equation}
\Theta=J^N + \tilde{K},
\end{equation}
where $J^N$ is defined in \eqref{JND.def}
and $\tilde{K}$ is again a Hilbert-Schmidt operator.
$J^N$~itself, however, can be a sum of a bounded and a Hilbert-Schmidt operator,
as we shall see in examples.
\end{remark}
\begin{proposition}\label{Omega.hol}
Let $\mathcal{S}$ be an open connected set in $\C^2$ such that for all $(c_-,c_+) \in \mathcal {S}$ all eigenvalues of $H$ are simple. Then $\Omega$ and thereby $\Theta$ are bounded holomorphic families in $\mathcal{S}$ with respect to parameters $c_{\pm}$.
\end{proposition}
\begin{proof}
We verify the criterion stated in \cite[Sec.~VII.1.1]{Kato-1966}.
We have proved already that $\Omega$ is bounded.
It remains to show that $\langle f, \Omega g\rangle$
is holomorphic for every $f,g$ from a fundamental set of $\H$
that we choose as the orthonormal basis $\{e_n\}_{n=0}^{\infty}$.
$\langle e_m, \Omega e_n\rangle = \langle \phi_m, e_n \rangle$
is holomorphic because $\phi_m$ is an eigenfunction of the operator~$H^*$,
which can be viewed as a holomorphic family of operators of type~$(B)$
with respect to the parameters~$c_{\pm}$.
\end{proof}
\begin{corollary}\label{h.hol}
Assume the hypothesis of Proposition \ref{Omega.hol}.
Then $h:=\Omega H \Omega^{-1}$ is a holomorphic family
of operators in $\mathcal{S}$ with respect to parameters $c_{\pm}$.
\end{corollary}
Since the operator $H$ is a holomorphic family of type $(B)$,
\ie~it is naturally defined via quadratic forms
with the domain $W^{1,2}(-a,a)$
independent of the parameters $c_{\pm}$,
$h$~is expected to possess a similar property.
To prove it, we have to particularly show that
the associated quadratic forms
corresponding to different values of~$c_\pm$
have the same domain,
which is not guaranteed by Corollary \ref{h.hol}.
To this end we analyse the quadratic form associated to $h$,
where we set $e_n:=\chi_n^N$ in the definition of $\Omega$.
\begin{theorem}\label{L*M}
Let all eigenvalues of $H$ be simple and let $e_n:=\chi_n^N$ in \eqref{Omega.def}.
Then $\Omega=I+L$ and $\Omega^{-1}=I+M$,
where $L$, $M$ are Hilbert-Schmidt operators.
$\Omega, \Omega^*, \Omega^{-1},(\Omega^{-1})^*$
are bounded operators on $W^{1,2}(-a,a)$ and $W^{2,2}(-a,a)$.
Furthermore, the following estimates hold for all
$\phi \in  W^{1,2}(-a,a)$ and arbitrary $\delta > 0$:
\begin{equation}\label{L*M.ineq}
\begin{aligned}
\|(L^* \phi)'\|^2 &
\leq C\left( \delta \;\! \|\phi'\|^2 +  \delta^{-2} \|\phi\|^2 \right), \\
\|(M \phi)'\|^2 &
\leq C\left( \delta \;\! \|\phi'\|^2 +  \delta^{-2} \|\phi\|^2 \right),
\end{aligned}
\end{equation}
with $C$ being constants not dependent on $\delta$ and $\phi$.
\end{theorem}
\begin{proof}
We set $a:=\pi/2$ as in the proof of Theorem~\ref{Omega.real}.
$M$ is Hilbert-Schmidt since $I=\Omega \Omega^{-1}=I + L + M + L M $
and $L$ is Hilbert-Schmidt.

We consider $\Omega^*$ at first. Following the proof of Theorem \ref{Omega.real},
$L^*$ can be written as
\begin{equation}\label{L*.def}
L^*f=\sum_{k=0}^{\infty} \tilde{\phi}_k \langle \chi_k^N, f \rangle,
\end{equation}
where $\tilde{\phi}_k:=\phi_k-\chi_k^N$ and $f\in \H$. We show that $L^*$ is bounded on $W^{1,2}(-a,a)$. We estimate the Hilbert-Schmidt norm of $L^*$ on $W^{1,2}(-a,a)$ with help of the orthonormal basis $f_n:=\chi_n^N/\sqrt{1+n^2}$. In fact, it suffices to estimate:
\begin{equation}\label{L*.HS}
\sum_{n=0}^{\infty} \langle (L^* f_n)', (L^* f_n)' \rangle = \sum_{n=0}^{\infty} \frac{1}{1+n^2} \| \tilde{\phi}'_n \|^2
\end{equation}
where (recall~\eqref{phi.rewritten} and~\eqref{cn.sn.def})
\begin{equation}\label{phin'}
\begin{split}
\tilde{\phi}'_n &= -n \overline{\varepsilon_n}^2\, \chi_n^D c_n - \overline{\varepsilon_n}^2\, \chi_n^N s_n
- n \overline{\varepsilon_n} \, \chi_n^N s_n - \overline{\varepsilon_n} \, \chi_n^D (1+ \overline{\varepsilon_n}^2 \, c_n ) \\
&\quad + \overline{c_-} \left[ \chi_n^N (1 + \overline{\varepsilon_n}^2 \, c_n ) - \overline{\varepsilon_n} \, \chi_n^D s_n  \right].
\end{split}
\end{equation}
Using the asymptotic properties of~$\varepsilon_n$ and
the uniform boundedness of~$c_n,s_n$
(see \eqref{ln.as} and \eqref{cn.sn.def}, respectively)
together with the normalization of $\chi_n^{\iota}$,
we conclude that $\|\tilde{\phi}'_n\| \leq C$
uniformly in~$n$.
Therefore \eqref{L*.HS} is finite.

Using the same technique,
we can show that the Hilbert-Schmidt norm
of~$L^*$ in $W^{2,2}(-a,a)$ is finite.
To this end we select the basis $\chi_n^N/\sqrt{1+n^2+n^4}$,
the rest is based on $\|\tilde{\phi}''_n\|=\mO(n)$ as $n\to\infty$.

Let us now establish the inequalities~\eqref{L*M.ineq}.
Consider $\phi \in W^{1,2}(-a,a)$,
its basis decomposition $\phi=\sum_{n=0}^{\infty}\alpha_n \chi_n^N$,
and the identity
\begin{equation}
\sum_{n=0}^{\infty}|n\alpha_n|^2 = \|\phi'\|^2.
\end{equation}
Hence,
\begin{equation}
\|(L^*\phi)'\|^2 = \sum_{m,n=0}^{\infty} \overline{\alpha_m} \alpha_n \langle \tilde{\phi}'_m, \tilde{\phi}'_n \rangle,
\end{equation}
and having the explicit form of $\tilde{\phi}'_n$, see \eqref{phin'},
we have to estimate several terms.
We show the technique only for one term,
the estimate of remaining terms is analogous.
First, using the uniform boundedness of~$\|c_n\|,\|s_n\|$,
the asymptotics $\varepsilon_n=\mathcal{O}(n^{-1})$
and the uniform boundedness of~$\|\chi_n^N\|_\infty$,
it is easy to see that
\begin{equation*}
\begin{split}
&\sum_{m,n=0}^{\infty} m\, n\, |\alpha_m| |\alpha_n| |\varepsilon_m| |\varepsilon_n| |\langle \chi_m^N s_m, \chi_n^N s_n  \rangle|  \leq
C \left(\sum_{n=1}^{\infty} |\alpha_n|\right)^2  \\
\end{split}
\end{equation*}
holds with some positive constant~$C$.
It remains to estimate the~$l^1$-norm of~$\alpha_n$
by the $l^2$-norms of~$\alpha_n$ and $n\alpha_n$
(which equal $\|\phi\|$ and $\|\phi'\|$, respectively).
This is rather algebraic:
\begin{equation*}
\begin{aligned}
  \left( \sum_{n=1}^{\infty} |\alpha_n| \right)^2
  &= \left(
  \sum_{n=1}^{\infty}  \big(|\alpha_n| \, n\big)^b \, |\alpha_n|^{1-b} \, n^{-b}
  \right)^2
  \\
  & \leq \left( \sum_{n=1}^{\infty} |\alpha_n|^2 \, n^2  \right)^{b}
  \left( \sum_{n=1}^{\infty} |\alpha_n|^{2} \right)^{1-b}
  \left( \sum_{n=1}^{\infty} n^{-2b} \right)
  \\
  & \leq C_b \, \|\phi'\|^{2b} \, \|\phi\|^{2(1-b)}
  \\
  & \leq C_b \left(
  b \, \delta \, \|\phi'\|^{2} + (1-b) \, \delta^{-\frac{b}{1-b}} \, \|\phi\|^{2}
  \right),
\end{aligned}
\end{equation*}
with any $b,\delta \in (0,1)$.
Here the first inequality follows by the generalized H\"older inequality
and the last one is a consequence of the Young inequality.
The exponent~$b$ is chosen in such a way that $2b>1$,
so that the sum of $n^{-2b}$ (denoted by~$C_b$) converges.
If we put $b=2/3$, we obtain the inequality in the claim.

One can show, using the asymptotics \eqref{ln.as}, that
it follows from the normalization requirement
$\langle \phi_n, \psi_n\rangle=1$ that $A_n$,
the normalization constants of $\psi_n$, see \eqref{H.EF},
satisfy $A_n=1+\mO(n^{-1})$.
Then the claims for $\Omega^{-1}$
and $M$ can be derived in the same manner.

To justify that $\Omega$ and $(\Omega^{-1})^*$ are bounded on
$W^{1,2}(-a,a)$ and $W^{2,2}(-a,a)$,
it suffices to realize that $\Omega^{-1}$ and $\Omega^*$ are invertible
because they are invertible in $L^2(-a,a)$
and the inverse is bounded because of the form identity
plus compact operator on considered Sobolev spaces.
\end{proof}

\begin{corollary}\label{cor.th}
Assume the hypotheses of Theorem~\ref{L*M}.
Then $h:=\Omega H \Omega^{-1}$
is a holomorphic family of operators of type~$(B)$
with respect to~$c_\pm$.
The associated quadratic form~$t_h$,
in the sense of the representation theorem \cite[Thm. VI.2.1]{Kato-1966}, reads
\begin{equation}\label{th.psi}
\begin{split}
t_h[\psi]& =\|\psi'\|^2 + \langle (L^*\psi)', \psi'\rangle + \langle \psi', (M\psi)'\rangle + \langle (L^*\psi)', (M\psi)'\rangle  \\
& \quad + c_+ \left[ \big(\overline{\psi(a)} + \overline{(L^*\psi)(a)}\big)
\big(\psi(a) + (M\psi)(a)\big)    \right] \\
& \quad - c_- \left[ \big(\overline{\psi(-a)} + \overline{(L^*\psi)(-a)}\big)
\big(\psi(-a) + (M\psi)(-a)\big)    \right], \\
\Dom(t_h)&=W^{1,2}(-a,a).
\end{split}
\end{equation}
\end{corollary}
\begin{proof}
The form~$t_h$ defined in~\eqref{th.psi} is sectorial and closed
due to the perturbation result \cite[Thm.~VI.1.33]{Kato-1966},
regarding $u[\psi]:=t_h[\psi]-\|\psi'\|^2$
as a perturbation of $t_0[\psi]:=\|\psi'\|^2$.
Indeed, the inequalities \eqref{L*M.ineq} applied on~$u[\psi]$
yield that~$u$ is $t_0$-bounded with $t_0$-bound~$0$.
Therefore, due to the first representation theorem
\cite[Thm.~VI.2.1]{Kato-1966},
there is a unique m-sectorial operator associated with~$t_h$.
Let us denote it by~$\tilde{h}$.
Our objective is to show that $\tilde{h}=h$.

Using the definition of~$h$ by the similarity transformation,
\ie\ $h=\Omega H \Omega^{-1}$,
and the fact that~$H$ is associated to~$t_H$,
we know that the domain of~$h$ are functions~$u$ such that,
firstly, $\Omega^{-1}u \in W^{1,2}(-a,a)$
and, secondly, there exists $w \in L^2(-a,a)$ such that
\begin{equation}
  t_H( \Omega^* v, \Omega^{-1} u ) = (v,w)
\end{equation}
for all $v$ such that $\Omega^* v \in W^{1,2}(-a,a)$.
However, by Theorem~\ref{L*M},
$\Omega$, $\Omega^*$, $\Omega^{-1}$, $(\Omega^*)^{-1}$
are bounded on $W^{1,2}(-a,a)$
and it is easy to check that
the identity
\begin{equation}
 t_H( \Omega^* v, \Omega^{-1} u ) = t_h(v,u)
\end{equation}
holds for all $u,v \in W^{1,2}(-a,a)$.
Consequently, the operators $\tilde{h}$ and $h$ indeed coincide.
\end{proof}
\begin{remark}
We remark that the boundedness of $\Omega$, $\Omega^*$, $\Omega^{-1}$
and $(\Omega^{-1})^*$ in $W^{2,2}(-a,a)$
was not used in the proof Corollary \ref{cor.th}.
Nevertheless, this property is useful if we analyse
the domain of $h$ directly from the relation $h=\Omega H \Omega^{-1}$.
It follows that $\Dom(h)$ consists of functions $\psi$ from $W^{2,2}(-a,a)$
satisfying boundary conditions
$(\Omega^{-1} \psi)'(\pm a) + c_{\pm}  (\Omega^{-1} \psi)(\pm a) = 0$.
\end{remark}
%

%-----------------------------------------------------------------------%
\section{Closed formulae in \texorpdfstring{$\PT$}{PT}-symmetric cases}
\label{sec:examples}
%-----------------------------------------------------------------------%

We present closed formulae of operators $\Theta$, $\Omega$ and~$h$
corresponding to~$H$
with special $\PT$-symmetric choice of boundary conditions,
$c_{\pm}:=\ii \alpha,$ with $\alpha \in \R$.
This case has already been studied in a similar context
in \cite{Krejcirik-2006-39, Krejcirik-2008-41a},
where the first formulae of the metric~$\Theta$ were given.
We substantially generalize these results here.

We essentially rely on the original idea of~\cite{Krejcirik-2008-41a}
to ``use the spectral theorem backward'' to sum up the infinite series
appearing in the definition of~$\Theta$ in~\eqref{Theta.def}.
The attempts to find $\Omega$ as the square root of $\Theta$
using the holomorphic and self-adjoint calculus are contained
in \cite{Zelezny-2011-50,Zelezny-MT},
however, only approximations of the resulting similar self-adjoint
operator~$h$ were found there.
The main novelty of the present approach comes from
the more general factorization~\eqref{Theta.dec} with~\eqref{Omega.def},
which enables us to obtain exact results.
Formulae contained in this section are obtained
by tedious although straightforward calculations
that we do not present entirely.

Finally, we present the metric operator for $H$
with general $\PT$-symmetric boundary conditions,
$c_{\pm}:=\ii \alpha \pm \beta$.
In this case, the eigenvalues are no longer explicitly known,
nevertheless, the experience from previous examples and formulation
of partial differential equation together
with a set of ``boundary conditions" for the kernel
of the integral operator provide the correct result.

%----------------------------------------------------%
\subsection{Reduction to finding a Neumann metric}
%----------------------------------------------------%
%
We start with the following fundamental result.
\begin{proposition}\label{Ex.Theta.0}
Let $c_{\pm}:=\ii \alpha,$ with $\alpha \in \R$.
Then the operator $\Theta$ defined in~\eqref{Theta.def}
has the form
\begin{equation}\label{Ex.Theta.1}
\Theta=J^N + C_0^2 \;\! \theta_1 + J^N\theta_2 + J^D \theta_3,
\end{equation}
where $J^{\iota}$, with $\iota\in\{D,N\}$,
are defined in \eqref{JND.def}, $C_0 > 0$,
and $\theta_i$ are integral operators with kernels:
\begin{equation}\label{Ex.Theta.2}
\begin{aligned}
\theta_1(x,y) &:= \frac{\ii}{a} \, e^{ \frac{\ii \alpha}{2}(x-y)} \sin\left( \frac{\alpha}{2}(x-y) \right), \\
\theta_2(x,y) &:= \frac{\ii \alpha}{2 a} \, \big[ y - a\, \sgn (y-x) \big],   \\
\theta_3(x,y) &:= \frac{\alpha^2}{2a} \left(  a^2 - xy \right)
- \frac{\ii \alpha}{2a} \, x - \frac{\ii \alpha}{2} \,
\big[ 1 - \ii \alpha (y-x)  \big] \sgn(y-x).
\end{aligned}
\end{equation}
$\Theta$ is the metric operator for $H$, see Definition \ref{Theta.definition},
if, and only if, $\alpha \neq k_n$ for every $n \in \N$.
\end{proposition}
\begin{proof}
Using the explicit form~\eqref{H.beta.0.EF} of functions $\phi_n$
and the definition~\eqref{Theta.def} of~$\Theta$,
we obtain
\begin{equation}\label{Ex.Theta.3}
\begin{aligned}
\Theta &= \sum_{n=0}^\infty C_n^2 \chi_n^N \langle \chi_n^N,\cdot \rangle + C_0^2 \left( \phi_0 \langle \phi_0 ,\cdot\rangle -  \chi_0^N \langle \chi_0^N,\cdot \rangle \right)
\\
& \quad + \alpha^2 \sum_{n=1}^\infty \frac{C_n^2}{k_n^2}\chi_n^D \langle  \chi_n^D,\cdot \rangle +\ii \alpha \sum_{n=1}^\infty \frac{C_n^2}{k_n}\chi_n^D \langle \chi_n^N,\cdot \rangle - \ii \alpha \sum_{n=1}^\infty \frac{C_n^2}{k_n} \chi_n^N \langle \chi_n^D,\cdot \rangle.
\end{aligned}
\end{equation}
Employing the operators $J^{\iota}$ and $p,p^*$ introduced in \eqref{JND.def}
and \eqref{pp*.def}, respectively,
and relations \eqref{p.ND.id} we obtain:
\begin{equation}\label{Ex.Theta.4}
\begin{split}
\Theta  & =  J^N  \sum_{n=0}^{\infty}\chi_{n}^{N} \langle \chi_{n}^{N},\cdot \rangle +
C_0^2 \left( \phi_0 \langle \phi_0 ,\cdot\rangle -  \chi_0^N \langle \chi_0^N,\cdot \rangle \right)  \\
& \quad + \alpha J^N  p \sum_{n=1}^{\infty} \frac{1}{k_{n}^{2}} \chi_{n}^{D} \langle \chi_{n}^{D},\cdot \rangle \\
& \quad + J^D \left( \alpha^{2} \sum_{n=1}^{\infty}\frac{1}{k_{n}^{2}} \chi_{n}^{D} \langle \chi_{n}^{D},\cdot \rangle +
\alpha p^* \sum_{n=1}^{\infty}\frac{1}{k_{n}^{2}}\chi_{n}^{N} \langle \chi_{n}^{N},\cdot \rangle \right).
\end{split}
\end{equation}
It follows from the functional calculus for self-adjoint operators
that \eqref{Ex.Theta.4} can be written as
\begin{equation}\label{operator.form}
\begin{aligned}
\Theta & =J^N +  C_0^2 \left( \phi_0 \langle \phi_0,\cdot \rangle -
\chi_0^N \langle \chi_0^N, \cdot \rangle \right)  +
\alpha J^N  p\;\!(-\Delta_D)^{-1}
\\
& \quad + J^D
\big[
\alpha^2 (-\Delta_D)^{-1}+ \alpha \;\! p^* (-\Delta_N^\perp)^{-1}
\big].
\end{aligned}
\end{equation}
By inserting the explicit integral kernels of the resolvents,
see Section \ref{subsec.not}, we obtain the formula~\eqref{Ex.Theta.1}
with~\eqref{Ex.Theta.2}.

To ensure that such~$\Theta$ represents as metric operator,
we recall that the spectrum of~$H$ is always real,
see Proposition~\ref{H.PT.EV}.
Moreover, it is simple if, and only if,
the condition in the last claim is satisfied.
\end{proof}
\begin{remark}
The formula \eqref{Ex.Theta.1} can be rewritten in terms of
the operator~$J^N$ only. Indeed, it is possible to show that
\begin{equation}
J^D = p^*J^N p \;\! (-\Delta_D)^{-1}.
\end{equation}
The final result is then
\begin{equation}\label{Ex.Theta.5}
\Theta= J^N + C_0^2 \theta_1 + J^N \theta_2 + p^*J^N \theta_4,
\end{equation}
where $\theta_4 := p \;\! (-\Delta_D)^{-1}\theta_3$
is an integral operator with kernel
\begin{equation}\label{theta4}
\begin{aligned}
\theta_4(x,y) & = \frac{\alpha}{12a}\bigg( y^2 (3-\ii \alpha y)
+3x^2 (1-\ii \alpha y)+2a^2 \big[1+\ii \alpha(3x - y)\big]\bigg) \\
& \quad - \frac{1}{4} \alpha \bigg( 2 - \ii \alpha(y-x)\bigg)(y-x) \sgn (y-x).
\end{aligned}
\end{equation}
Note that the expression~\eqref{theta4} is a result
of a rather lengthy computation.
\end{remark}

Any metric operator for $H$ in Proposition \ref{Ex.Theta.0}
can be obtained by determining~$J^N$ for given constants $C_n$.
Thus we managed to transform the problem of constructing the metric operators
for non-self-adjoint operator $H$ to the problem of constructing
the metric operators $J^N$ for the Neumann Laplacian $-\Delta_N$.
This significantly simplifies the problem,
since $-\Delta_N$ is self-adjoint and its metric operators are bounded,
positive operators with bounded inverse commuting with $-\Delta_N$.
For instance, any bounded, uniformly positive function of~$-\Delta_N$
represents a metric operator.
Moreover, it was shown in \cite{Zelezny-MT}
that any~$J^N$ can be approximated in the strong sense
by a polynomial of $I + \lambda (-\Delta_N-\lambda)^{-1}$,
with $\lambda \in \rho(-\Delta_N)$.

We consider two choices of constants $C_n$ in the following
and we find final formulae for the corresponding metric operators.

\subsection{The constant-coefficients metric}
Let $C_n^2:=1$ for every $n \geq 0$.
Then $J^N=J^D = I$ and the metric operator $\Theta$ reads $\Theta =  I + K$,
where $K$ is an integral operator with the kernel
\begin{equation}\label{metric-constant}
\begin{split}
\mathcal{K}(x,y) & =
\frac{\ii}{a} \, e^{\ii \frac{\alpha}{2}(x-y)}
\sin\left( \frac{\alpha}{2}(x-y) \right) +
 \frac{\ii \alpha}{2a} \, \big( |y-x| -2a  \big)
 \sgn (y-x)
 \\
& \quad + \frac{\alpha^2}{2a} \;\! \left( a^2-xy-a|y-x| \right).
%
%
% \frac{\ii}{a} e^{\ii \frac{\alpha}{2}(x-y)} \sin\left( \frac{\alpha}{2}(x-y) \right) +
%\frac{\alpha^2}{2a} \left( a^2-xy \right) + \frac{\ii \alpha}{2a}(y-x) \\
%& \quad - \frac{\ii \alpha}{2} \left( 2 - \ii \alpha (y-x)  \right) \sgn (y-x).
\end{split}
\end{equation}
Formula~\eqref{metric-constant} represents a remarkably
elegant form for the metric operator
found firstly in \cite{Krejcirik-2006-39, Krejcirik-2008-41a}.

\subsection{The \texorpdfstring{$\mC$}{C} operator}
Another choice of $C_n$ is motivated by the concept of $\mC$ operator,
see Definition~\ref{c.def}.
We want to find such $\Theta$ that $\mC^2=I$, where $\mC=\P\Theta$.
Since $H$ is $\P$-self-adjoint, we have $\P \phi_n=D_n \psi_n$
with some numbers~$D_n$.
Assuming the non-degeneracy condition
$\alpha \neq k_n$ for every $n \geq 0$,
an explicit calculation shows that
\begin{equation}
D_0 = \frac{\sin(2\alpha a)}{2\alpha a}, \ \
D_n= (-1)^n\,\frac{k_n^2-\alpha^2}{k_n^2}, \ n \in \N.
\end{equation}
The condition $(\P\Theta)^2=I$ then restricts~$C_n$
from~\eqref{Theta.def} to
\begin{equation}\label{C-choice}
C_0^2 = \frac{2|\alpha|a}{|\sin(2\alpha a)|}, \ \ C_n^2= \frac{k_n^2}{|k_n^2-\alpha^2|}, \ n \in \N.
\end{equation}
In order to simplify the formulae,
we consider only $\alpha \in (0,k_1)$ in the following.

\begin{remark}
As mentioned below~\eqref{Theta.def}, any choice of~$C_n$
can be interpreted as a sort of normalisation of~$\phi_n$.
It is therefore interesting to notice that~\eqref{C-choice}
results into the symmetric normalization of $\phi_n$ and $\psi_n$
when $\langle \phi_n,\psi_n\rangle=1$ is required:
\begin{align*}
\psi_0(x)
& = \sqrt{\frac{\alpha}{\sin(2\alpha a)}} \,
e^{\ii \alpha a} e^{- \ii \alpha x},
&
\psi_n(x)
& = \frac{k_n}{\sqrt{k_n^2-\alpha^2}}
\left(  \chi_n^N(x) - \ii \frac{\alpha}{k_n} \chi_n^D(x) \right), \\
\phi_0(x)
& = \sqrt{\frac{\alpha}{\sin(2\alpha a)}} \,
e^{\ii \alpha a} e^{  \ii \alpha x},&
\phi_n(x)
& = \frac{k_n}{\sqrt{k_n^2-\alpha^2}}
\left(  \chi_n^N(x) + \ii \frac{\alpha}{k_n} \chi_n^D(x) \right).
\end{align*}
These expressions should be compared with
the normalization of
\eqref{H.beta.0.EF}--\eqref{H.beta.0.EF2},
standardly used in the present paper.
The symmetric form of the ``present normalization''
indicates that the choice~\eqref{C-choice} will lead
to a simpler form of~$\Theta$ than~\eqref{metric-constant}.
\end{remark}

Using~\eqref{C-choice} in the series~\eqref{JND.def},
the operators $J^{\iota}$ can be determined by the functional calculus:
\begin{equation}\label{J.C-choice}
\begin{aligned}
J^N& =\sum_{n=0}^{\infty} \frac{k_n^2}{k_n^2-\alpha^2} \,
\chi_n^N \langle\chi_n^N,\cdot \rangle
+ C_0^2 \, \chi_0^N \langle \chi_0^N,\cdot\rangle
\\
& = (-\Delta_N)(-\Delta_N-\alpha^2)^{-1}
+C_0^2\,\chi_0^N \langle \chi_0^N,\cdot \rangle
\\
& = I+\alpha^2(-\Delta_N-\alpha^2)^{-1}
+C_0^2 \, \chi_0^N \langle \chi_0^N,\cdot \rangle, \\
J^D & =\sum_{n=1}^{\infty} \frac{k_n^2}{k_n^2-\alpha^2} \,
\chi_n^D \langle \chi_n^D,\cdot \rangle
\\
&=(-\Delta_D)(-\Delta_D-\alpha^2)^{-1}
\\
&= I+\alpha^2(-\Delta_D-\alpha^2)^{-1}.
%=(-\Delta_D)(-\Delta_D-\alpha^2)^{-1}
%\\
%& = I+\alpha^2(-\Delta_D-\alpha^2)^{-1}.
\end{aligned}
\end{equation}
A direct (but very tedious) way how to derive the metric~$\Theta$
for the choice~\eqref{C-choice} is to express
the resolvents of the Dirichlet and Neumann Laplacians
from the ultimate expressions in~\eqref{J.C-choice}
by means of the Green's functions~\eqref{res.DN.ker}
and compose them with the operators $\theta_i$ in~\eqref{Ex.Theta.1}.

However, a more clever way how to proceed
is to come back to the operator form~\eqref{operator.form}
and perform first some algebraic manipulations
with the intermediate expressions appearing in~\eqref{J.C-choice}.
First, we clearly have
$
  J^D (-\Delta_D)^{-1} = (-\Delta_D-\alpha^2)^{-1}
$.
Second, employing~\eqref{pp*.def} and the identity
$
  (-\Delta_N) (-\Delta_N^\perp)^{-1}
  = I - \chi_0^N \langle\chi_0^N,\cdot \rangle
$,
we check
\begin{equation*}
  \left[ J^D p^* (-\Delta_N^\perp)^{-1} \right]^*
  = p \;\! (-\Delta_D-\alpha^2)^{-1}
  ,
  \quad
  \left[ J^N p \;\! (-\Delta_D)^{-1} \right]^*
  = p^* (-\Delta_N-\alpha^2)^{-1}
  .
\end{equation*}
Finally, again using~\eqref{pp*.def},
we verify the intertwining relation
$
  [p \;\! (-\Delta_D-\alpha^2)^{-1}]^*
  = p^* (-\Delta_N-\alpha^2)^{-1}
$.
Summing up, with our choice~\eqref{C-choice},
formula~\eqref{operator.form} simplifies to
\begin{equation}
\begin{aligned}
\Theta =& \ I +  C_0^2 \, \phi_0 \langle \phi_0,\cdot \rangle
+ \alpha^2 (-\Delta_N-\alpha^2)^{-1}
+ \alpha^2 (-\Delta_D-\alpha^2)^{-1}
\\
& + \alpha \;\! p \;\! (-\Delta_D-\alpha^2)^{-1}
+ \alpha \;\! p^* (-\Delta_N-\alpha^2)^{-1}
.
\end{aligned}
\end{equation}
Now it is easy to substitute~\eqref{res.DN.ker}
and after elementary manipulations
to conclude with $\Theta = I + K$,
where $K$ is an integral operator with the kernel
\begin{equation}\label{K.ker.C}
\begin{aligned}
\mathcal{K}(x,y) & = \alpha \, e^{-\ii \alpha (y-x)} \,
\big[ \tan(\alpha a) - \ii \sgn (y-x) \big].
\end{aligned}
\end{equation}
The operator $\mC$ can be found easily by composing $\P$ and $\Theta$.
We finally arrive at the formula $ \mC = \P + L $,
where $L$ is an integral operator with the kernel
\begin{equation}\label{K.ker.CC}
\begin{aligned}
\mathcal{L}(x,y) & = \alpha \, e^{-\ii \alpha (y+x)} \,
\big[ \tan(\alpha a) - \ii \sgn (y+x) \big].
\end{aligned}
\end{equation}

\subsection{The similar self-adjoint operator}
Next we present an example of operator $\Omega$, defined in \eqref{Omega.def}
with $e_n:=\chi_n^N$, that will be used to find the similar self-adjoint
operator~$h$ from~\eqref{sim-intro}.
We recall that the similarity transformation~$\Omega$ is invertible
if all the eigenvalues of $H$ are simple,
which is ensured by the condition
$\alpha \neq k_n$ for every $n\in\N$.
We will actually search for the quadratic form associated to $h$
for which we have the result in Corollary \ref{cor.th}.

We follow the analogous strategy to obtain formula for $\Omega$
as in the proof of Proposition \ref{Ex.Theta.0}.
The definition of $\Omega$ with $e_n:=\chi_n^N$ leads to the sum:
\begin{equation}\label{Ex.Omega.1}
\begin{aligned}
\Omega & = \chi_0^N \langle \phi_0, \cdot \rangle + \sum_{n=1}^{\infty} \chi_n^N \langle \chi_n^N, \cdot \rangle - \ii \alpha \sum_{n=1}^{\infty} \frac{1}{k_n} \chi_n^N \langle \chi_n^D, \cdot \rangle  \\
& = I + \chi_0^N \langle \phi_0, \cdot \rangle - \chi_0^N \langle \chi_0^N, \cdot \rangle + \alpha p \sum_{n=1}^{\infty} \frac{1}{k_n^2} \chi_n^D \langle \chi_n^D, \cdot \rangle \\
& = I + \chi_0^N \langle \phi_0, \cdot \rangle
- \chi_0^N \langle \chi_0^N, \cdot \rangle + \alpha p \;\!  (-\Delta_D)^{-1},
\end{aligned}
\end{equation}
where we have used identities \eqref{p.ND.id}.
In the same manner,
%with omitting the details,
we obtain the result for the inverse~$\Omega^{-1}$:
\begin{equation}\label{Ex.Omega.2}
\begin{aligned}
\Omega^{-1} & = \psi_0 \langle \chi_0^N, \cdot \rangle + \sum_{n=1}^{\infty} \frac{k_n^2}{k_n^2-\alpha^2} \chi_n^N \langle \chi_n^N, \cdot \rangle - \ii \alpha \sum_{n=1}^{\infty} \frac{k_n}{k_n^2-\alpha^2} \chi_n^D \langle \chi_n^N, \cdot \rangle
\\
& = I + \psi_0 \langle \chi_0^N, \cdot \rangle + \alpha^2 (-\Delta_N-\alpha^2)^{-1} - \alpha p^* (-\Delta_N-\alpha^2)^{-1}.
\end{aligned}
\end{equation}
The operators $L$, $M$ appearing in the expressions
for $\Omega=I+L$ and $\Omega^{-1}=I+M$ are, as expected,
integral operators with the kernels $\mathcal{L},$ $\mathcal{M}$
that can be easily obtained using formulae
for the Neumann and Dirichlet resolvents \eqref{res.DN.ker}--\eqref{res.D0}:
\begin{equation}\label{LM.ker}
\begin{aligned}
\mathcal{L}(x,y)  & =  \frac{\ii \alpha}{2a} \,
\big[ y- a \sgn (y-x) \big]
+\frac{1}{2a} \left( e^{ - \ii \alpha (y+a)} -1  \right),  \\
\mathcal{M}(x,y)  & = \frac{\alpha e^{\ii \alpha(a-x)} }{\sin (2\alpha a)}
- \frac{\alpha}{2} \, e^{-\ii \alpha(x-y)}
\big[ \cot (2\alpha a) - \ii \sgn(y-x) \big]  \\
&  \quad - \frac{\alpha e^{-\ii \alpha (x+y)}}{2 \sin(2\alpha a)} .
\end{aligned}
\end{equation}

To find the similar self-adjoint operator~\eqref{sim-intro},
we start from the quadratic form~\eqref{th.psi}.
Inserting \eqref{LM.ker} into the latter and performing several
integrations by parts with noticing
that $LM=-L-M$ and $(M\psi)'=-\ii \alpha M \psi - \ii \alpha \psi$ results in:
\begin{equation}
t_h[\psi]= \|\psi'\|^2 + \alpha^2 | \langle \chi_0^N, \psi \rangle |^2.
\end{equation}
The corresponding operator $h$ reads:
\begin{equation}
\begin{aligned}
h\psi &= -\psi'' + \alpha^2 \chi_0^N \langle \chi_0^N, \psi \rangle, \\
\Dom(h) & = \left\{ \psi \in W^{2,2}(-a,a) : \psi'(\pm a)=0 \right\}.
\end{aligned}
\end{equation}
We remark that $h$ is a rank one perturbation of the Neumann Laplacian.
The eigenfunctions of $h$ are $\chi_n^N$ with $\chi_0^N$ corresponding
to the eigenvalue $\alpha^2$.

It is interesting to compare the spectra of $H$ and $h$ for $\alpha=k_n$,
\ie~in the points where the spectra are degenerate
and similarity transformation breaks down because the operator~$\Omega$
is not invertible. $k_n^2$ is an eigenvalue
with the algebraic multiplicity two for both $H$ and $h$.
However, the geometric multiplicity differs: it is one for $H$ and two for $h$.

The form of~$h$ also explains
the origin of the peculiar $\alpha$-dependence
of the eigenvalues of $H$
(which are all constant except for~$\lambda_0(\alpha)=\alpha^2$).
In fact, it is the nature of the rank one perturbation
to leave all the Neumann eigenvalues untouched
except for the lowest one that is driven to the $\alpha^2$ behavior.

\subsection{More general boundary conditions}
Finally, we consider the general $\PT$-symmetric boundary conditions
$c_{\pm}:=\ii \alpha \pm \beta$, with $\alpha,\beta \in \R$.
We start with formal considerations.
The $\Theta$-self-adjointness of $H$ can be
expressed in the following way.
We take the advantage of the realization of $\Theta=I + K$,
which we insert into $\Theta H \psi = H^* \Theta \psi$, $\psi \in \Dom(H)$.
A formal interchange of differentiation with integration and integration by parts yield following problem that we can understand in distributional sense:
\begin{align}
(\partial_x^2 - \partial_y^2) \mathcal{K}(x,y) &=0, \label{K.wave.eq}\\
\partial_y \mathcal{K}(x, \pm a) + (\ii \alpha \pm \beta) \mathcal{K}(x, \pm a) & = 0. \label{K.BC.1}
\end{align}
Moreover, $\Theta \psi$ must belong to $\Dom(H^*)$, from which we have a condition
\begin{equation}\label{K.BC.2}
\partial_x \mathcal{K} (\pm a, y) + (-\ii \alpha \pm \beta)
\mathcal{K}(\pm a, y) = 2\ii \alpha \delta (y\mp a).
\end{equation}
Here~$\delta$ denotes the Dirac delta function.

Already presented examples of $\Theta$ for $\beta=0$
satisfy these requirements, particularly $\mathcal{K}$
solves the wave equation \eqref{K.wave.eq}.
The kernel \eqref{K.ker.C}, corresponding to the simpler form
of presented metric operators, is a function of $x-y$ only.
Inspired by this, we find the solution of the wave equation
\begin{equation}\label{K.beta}
\mathcal{K}(x,y)=e^{\ii \alpha (x-y) - \beta |x-y| } \,
\big[c+\ii \alpha  \sgn(x-y) \big], \ \ c \in \R,
\end{equation}
that satisfies the ``boundary conditions'' \eqref{K.BC.1}
and \eqref{K.BC.2} as well.
The one parametric family of solutions
\eqref{K.beta} of \eqref{K.wave.eq}--\eqref{K.BC.2}
demonstrates the known non-uniqueness of solutions to this problem.
We also remark that~$c$ can be taken as $\alpha$ or $a$ dependent as well.

The positivity of $\Theta$ is ensured if the norm of $K$ is smaller than 1.
This can be estimated by the Hilbert-Schmidt norm of $K$
which is explicitly computable:
\begin{equation}
\|K\|_{\rm HS}^2 = (c^2+\alpha^2) \,
\frac{4 a \beta + e^{-4 a \beta} -1 }{2 \beta^2}
\,.
\end{equation}
Consequently, the positivity of $\Theta$
can be achieved by several ways,
\eg, if~$a$ is small;
or if $\beta$ is positive and large;
or $|c|$ and $|\alpha|$ are small.
In any of the regimes,
the formal manipulations above are justified.
%

%------------------------------%
\section{Bounded perturbations}
\label{sec:bound.pert}
%------------------------------%
%
\label{bound.pert}

In this section we show that results of Section \ref{sec.metric} remain valid
if we consider a bounded perturbation $V$ of $H$.

Firstly we remark that the perturbation result \cite[Thm. XIX 2.7]{DS3}
guarantees that $H+V$ remains a discrete spectral operator.
That is, if all the eigenvalues of $H+V$ are simple,
then the metric operator $\Theta$ exists.
We show that the claim of Theorem \ref{Omega.real} is valid for $H+V$ as well.
The rest of the results from Section \ref{sec.metric} then follows straightforwardly.

Our approach is to use analytic perturbation theory for the operator
$h:=\Omega H \Omega^{-1}$ that is perturbed
by a bounded operator $\Omega V \Omega^{-1}$.
We denote by $\xi_n$, $\eta_n$ the eigenfunctions
of $H+V$ and $H^*+V^*$, respectively.
Let $e_n$ be elements of any orthonormal basis in~$\H$.
\begin{theorem}\label{Thm.bounded}
Let all the eigenvalues of $H$ be simple and let $V$ be a bounded operator.
If all eigenvalues of $H+V$ are simple,
then $\Omega_V = \sum_{n=0}^{\infty} e_n \langle \eta_n, \cdot \rangle$, \ie~$\Omega_V: \xi_n \mapsto e_n$,
can be expressed as
\begin{equation}
\Omega_V = U + L,
\end{equation}
where $U$ is a unitary operator and $L$ is a Hilbert-Schmidt operator.
\end{theorem}
\begin{proof}
As in the proof of Theorem \ref{Omega.real}, without loss of generality,
we restrict ourselves to $e_n:=\chi_n^N$
and we show that $\Omega_V=I+L$ with $L$ being Hilbert-Schmidt.
We consider the normal operator $h:=\Omega H \Omega^{-1}$
and we perturb it by $v:= \Omega V \Omega^{-1}$.
More specifically,
we construct $h(\varepsilon):=h + \varepsilon \, v$
forming a holomorphic family of type~$(A)$
with respect to the parameter $\varepsilon$.
We denote by $\mu_n(\varepsilon),$ $\overline{\mu_n(\varepsilon)}$ the eigenvalues and by $\tilde{\xi}_n (\varepsilon)$, $\tilde{\eta}_n(\varepsilon)$ the corresponding eigenfunctions of $h(\varepsilon)$ and of $h(\varepsilon)^*$ respectively. $h(0)$, $h(0)^*$ are normal, therefore the eigenfunctions $\tilde{\xi}_n (0)$ and $\tilde{\eta}_n(0)$ form orthonormal bases. In fact, with our choice of $e_n$, $\tilde{\xi}_n (0) = \tilde{\eta}_n(0) = \chi_n^N$.

We construct operator $\tilde{\Omega}: \tilde{\xi}_n(1) \mapsto \chi_n^N$ and we will show that $\tilde{\Omega}=I+\tilde{L}$, where $\tilde{L}$ is Hilbert-Schmidt. $\Omega_V$ is the composition of $\Omega$ and $\tilde{\Omega}$ and the claim then follows easily using of the fact that Hilbert-Schmidt operators are a *-both-sided ideal.

The distance of $\mu_n(0)$ and $\mu_n(1)$ can be at most $\|v\|$. Since we know the asymptotics of $\mu_n(0)=\lambda_n$, see \eqref{ln.as}, it is clear that there exists $N_0$ such that for all $n > N_0$,  $|\mu_{n+1}(1) - \mu_{n}(1)| > n$ holds. Moreover, for such $n$ the radius of convergence of perturbation series for eigenvalues and eigenfunctions is larger than 1. Thus, we have
\begin{equation}\label{eta.nj}
\tilde{\eta}_n(\varepsilon) = \chi_n^N + \sum_{j=1}^{\infty} \tilde{\eta}_n^{(j)} \varepsilon^j.
\end{equation}
We estimate the norms of $\tilde{\eta}_n^{(j)}$ using the analytic perturbation theory:
\begin{equation}\label{eta.nj.est}
\begin{aligned}
\|\tilde{\eta}_n^{(j)}\| & \leq \frac{1}{2\pi} \oint_{\Gamma_n} \left\| (h(0)^*-E)^{-1}( v^* (h(0)^*-E)^{-1})^{j}  \chi_n^N \right\| \dd E  \\
& \leq \frac{1}{2\pi} \oint_{\Gamma_n} \frac{ 2^{j+1} \| v \|^{j}}{n^{j+1}} \dd E \leq  \frac{c^{j}}{n^{j}},
\end{aligned}
\end{equation}
where $\Gamma_n$ is a circle around $\mu_n(0)$ of radius $n/2$ and the constant $c$ does not depend on $n$. We define $N_1$ as such that $N_1 \geq N_0$ and $c/N_1 < 1$.

We prove that $\tilde{\Omega}$ has the desired form by showing that the adjoint $\tilde{\Omega}^*=\sum_{n=0}^{\infty} \tilde{\eta}_n(1) \langle \chi_n^N , \cdot \rangle$ can be written as $\tilde{\Omega}^* = I + \tilde{L}^*_{N_1} + \tilde{L}^*_{\infty}$, where
\begin{equation}
\tilde{L}^*_{N_1}  := \sum_{n=0}^{N_1-1} (\tilde{\eta}_n(1) - \chi_n^N) \langle \chi_n^N, \cdot \rangle, \ \
\tilde{L}^*_{\infty}  := \sum_{n=N_1}^{\infty} \sum_{j=1}^{\infty} \tilde{\eta}_n^{(j)} \langle \chi_n^N, \cdot \rangle,
\end{equation}
and $\tilde{L}^*_{N_1}$ and $\tilde{L}^*_{\infty}$ are Hilbert-Schmidt.
The decomposition of $\tilde{\Omega}^*$ follows immediately if we consider the expansions \eqref{eta.nj} for $n>N_1$ and rewrite $\tilde{\eta}_n(1)=\chi_n^N + (\tilde{\eta}_n(1)-\chi_n^N)$ for $n \leq N_1$.  $\tilde{L}^*_{N_1}$ is a finite rank operator therefore it is obviously Hilbert-Schmidt. $\tilde{L}^*_{\infty}$ is bounded and the defining sum is absolutely convergent since
\begin{equation}
\begin{aligned}
\sum_{n=N_1}^{\infty} \sum_{j=2}^{\infty} \|\tilde{\eta}_n^{(j)} \| | \langle \chi_n^N, \psi \rangle |
\leq \|\psi \| \sum_{n=N_1}^{\infty} \sum_{j=2}^{\infty} \left( \frac{c}{n} \right)^j
\leq \|\psi \| \sum_{n=N_1}^{\infty} \frac{c^2}{n^2 - nc} ,
\\
\sum_{n=N_1}^{\infty} \|\tilde{\eta}_n^{(1)} \| | \langle \chi_n^N, \psi \rangle |
\leq c \sqrt{\sum_{n=N_1}^{\infty} \frac{1}{n^2}}
\sqrt{\sum_{n=N_1}^{\infty} | \langle \chi_n^N, \psi \rangle |^2}
\leq c \|\psi\| \sqrt{\sum_{n=N_1}^{\infty} \frac{1}{n^2}}.
\end{aligned}
\end{equation}
Finally we estimate the Hilbert-Schmidt norm of $\tilde{L}^*_{\infty}$:
\begin{equation}
\begin{aligned}
&\sum_{p=0}^{\infty}
\left \langle
\sum_{m=N_1}^{\infty} \sum_{i=1}^{\infty} \tilde{\eta}_m^{(i)} \langle \chi_m^N, \chi_p^N \rangle,
\sum_{n=N_1}^{\infty} \sum_{j=1}^{\infty} \tilde{\eta}_n^{(j)} \langle \chi_n^N, \chi_p^N \rangle
\right \rangle \\
&\leq
\sum_{p=N_1}^{\infty} \sum_{i=1}^{\infty} \left( \frac{c}{p} \right)^{i} \sum_{j=1}^{\infty} \left( \frac{c}{p} \right)^{j}
\leq \sum_{p=N_1}^{\infty} \left( \frac{c}{p-c} \right)^2 < \infty.
\end{aligned}
\end{equation}
This concludes the proof of the theorem.
\end{proof}
\begin{remark}[General Sturm-Liouville operators]
Let us conclude this section by a remark on how to extend
the previous result on bounded perturbations~$V$
for the operator~$H$ in the general form
$$
  H\psi := -(\rho\psi')' + V\psi
  \qquad \mbox{on} \qquad
  L^2(-a,a)
  \,,
$$
subject to the boundary conditions
\begin{equation}\label{bc-general}
\rho(\pm a) \psi'(\pm a) + c_{\pm} \psi(\pm a) = 0.
\end{equation}
Assuming merely that~$\rho$ is a bounded and uniformly positive function,
\ie, there exists a positive constant~$C$ such that
$C^{-1}\leq\rho(x) \leq C$ for all $x \in (-a,a)$,
the operator can be defined (\cf~\cite[Corol.~4.4.3]{Davies1995})
as an m-sectorial operator
associated with a closed sectorial form with domain $W^{1,2}(-a,a)$.
If, in addition, we assume that $\rho\in W^{1,\infty}(-a,a)$,
then it is possible to check that the domain of~$H$
consists of functions~$\psi$ from the Sobolev space $W^{2,2}(-a,a)$
satisfying~\eqref{bc-general}.

Now, let us strengthen the regularity
hypothesis to $\rho\in W^{2,\infty}(-a,a)$
and introduce the unitary (Liouville) transformation
$\mathcal{U} : L^2(-a,a) \to L^2(f(-a),f(a))$ by
$$
  \mathcal{U}^{-1}\phi := \rho^{-1/4} \, \phi \circ f
  \,,
  \qquad \mbox{where} \qquad
  f(x) := \int_0^x \frac{d\xi}{\sqrt{\rho(\xi)}}
  \,.
$$
Then it is straightforward to check that
the unitarily equivalent operator $\tilde{H}:=\mathcal{U}H\mathcal{U}^{-1}$
on $L^2(f(-a),f(a))$ satisfies
\begin{equation*}
\begin{aligned}
\tilde{H} \phi & = -\phi'' + \tilde{V}\phi + W\phi, \\
\Dom(\tilde{H}) & = \big\{ \phi \in W^{2,2}\big(f(-a),f(a)\big): \
\phi'(\pm f(a)) + \tilde{c}_\pm \phi(\pm f(a))=0 \big\},
\end{aligned}
\end{equation*}
where $\tilde{V} := \mathcal{U}V\mathcal{U}^{-1}$ and
$$
\begin{aligned}
  \tilde{c}_\pm := \frac{c_\pm}{\rho(\pm a)^{1/4}}
  - \frac{1}{4} \frac{\rho'(\pm a)}{\rho(\pm a)^{1/2}}
  \,, \qquad
  W := \left(\frac{1}{4} \rho''
  - \frac{1}{16}
  \frac{\rho'^2}{\rho}
  \right)
  \circ f^{-1}
  \,.
\end{aligned}
$$
In this way, we have transformed the second-order perturbation
represented by~$\rho$ into a bounded potential~$W$
and modified boundary conditions.
Theorem~\ref{Thm.bounded} applies to~$\tilde{H}$
and, as a consequence of the unitary transform~$\mathcal{U}$,
to~$H$ as well.
\end{remark}

\section{Conclusions}\label{Sec.end}

In this article, we investigated the structure of
similarity transformations~$\Omega$ and metric operators~$\Theta$
for Sturm-Liouville operators with separated,
Robin-type boundary conditions.
The main result is that~$\Omega$ and~$\Theta$
can be expressed as a sum of the identity and an integral Hilbert-Schmidt operator.

We would like to emphasize that this not always the case
for other types of operators,
see, \eg, \cite{Albeverio-2005-38, Siegl-2008-41, Kuzhel-2011-379, Gunther-2010-43},
where~$\Theta$ is a sum of the identity and a bounded non-compact operator.
The latter is a composition of the parity and the multiplication
by~${\rm sign}$ function.
Moreover, corresponding similarity transformations map (non-self-adjoint)
point interactions to (self-adjoint) point interactions,
which is not typically the case for operators studied here.
This is illustrated in the example of $\PT$-symmetric boundary
conditions where the equivalent self-adjoint operator
is not a point interaction but rather a rank one perturbation
of the Neumann Laplacian.

In this work we considered the separated boundary conditions only.
Nonetheless, the analogous results are expected to be valid
for all strongly regular boundary conditions.

As the proofs of the results show, the crucial property is the asymptotics
of eigenvalues, \ie~separation distance of eigenvalues tends to infinity,
that is used for the proof of the existence of similarity transformations~\cite{DS3}.
Recent results on basis properties for perturbations of harmonic oscillator
type operators \cite{Adduci-2009,Shkalikov-2010-269,Albeverio-2009-64}
give a possibility to investigate the structure of similarity transformation
in these cases as well. Another step is to extend the results \eg~on Hill operators,
where a criterion on being spectral operator of scalar type has been obtained in \cite{Gesztesy-2009-107} and recently extended in \cite{Djakov-2011}.

On the other hand,
the structure of similarity transformations for operators with continuous spectrum
as well as for multidimensional Schr\"odinger operators is almost unexplored
and constitutes thus a challenging open problem.

We illustrated the results by an example of $\PT$-symmetric boundary conditions,
where we found all the studied objects in a closed formula form,
which is hardly the case in more general situations.
However, in general, we may search for approximations of $\Omega$ or $\Theta$,
typically applying the analytic perturbation theory to find perturbation
series for eigenvalues and eigenfunctions of $H$ to certain order $k$.
For instance, we perturb the parameters $c_{\pm}$ in boundary conditions
by small $\varepsilon$.
As a result we find an approximation $h_{\rm app}$ of the similar
operator $h$ with resolvents satisfying
$\|(h - z)^{-1} - (h_{\rm app} - z)^{-1} \| \leq C \varepsilon^k$.
 An extensive discussion and example of such construction
can be found in \cite{Zelezny-MT}.
The same remark is appropriate for small perturbations
by bounded operator discussed
in Section~\ref{bound.pert}.

\section*{Acknowledgement}
D.K.\ acknowledges the hospitality of
the Deusto Public Library in Bilbao.
This work has been partially supported by
the Czech Ministry of Education, Youth, and Sports
within the project LC06002 and by the GACR grant No. P203/11/0701.
P.S.\ appreciates the support by
GACR grant No. 202/08/H072 and by the Grant Agency of the Czech Technical University in Prague,
grant No. SGS OHK4-010/10.
J.\v{Z}.\ appreciates the support by
the Czech Ministry of Education, Youth, and Sports
within the project LC527.

%\addcontentsline{toc}{section}{References}
{\small
%\bibliographystyle{acm}
%\bibliography{references}

}

\end{document}